\documentclass[12pt]{amsart}
\date{\today}

\input xymatrix
\xyoption{all}

\usepackage{amsmath,amsthm,amsfonts,amssymb,mathrsfs}

\usepackage{color}

\usepackage{amssymb, cite}
\usepackage[colorlinks,plainpages,citecolor=magenta, linkcolor=blue , backref]{hyperref}

\usepackage{hyperref}

\newtheorem{theorem}{Theorem}[section]
\newtheorem{proposition}[theorem]{Proposition}
\newtheorem{corollary}[theorem]{Corollary}
\newtheorem{lemma}[theorem]{Lemma}

\theoremstyle{definition}
\newtheorem{example}[theorem]{Example}
\newtheorem{remark}[theorem]{Remark}
\newtheorem{question}[theorem]{Question}

\begin{document}

\title[On the monoid of cofinite partial isometries of $\mathbb{N}$  with the usual metric]{On the monoid of cofinite partial isometries of $\mathbb{N}$ with the usual metric}

\author[O.~Gutik and A.~Savchuk]{Oleg~Gutik and Anatolii~Savchuk}
\address{Faculty of Mathematics, Ivan Franko National
University of Lviv, Universytetska 1, Lviv, 79000, Ukraine}
\email{oleg.gutik@lnu.edu.ua, ovgutik@yahoo.com, asavchuk3333@gmail.com}

\keywords{Partial isometry, inverse semigroup, partial bijection, bicyclic monoid, embedding, group congruence, generator}

\subjclass[2010]{20M18, 20M20, 20M30}

\begin{abstract}
In the paper we show that the monoid $\mathbf{I}\mathbb{N}_{\infty}$ of all partial cofinite isometries of positive integers  does not embed isomorphically into the monoid $\mathbf{ID}_{\infty}$ of all partial cofinite isometries of integers.  Moreover, every non-annihilating homomorphism $\mathfrak{h}\colon \mathbf{I}\mathbb{N}_{\infty}\to\mathbf{ID}_{\infty}$ has the following property:   the image $(\mathbf{I}\mathbb{N}_{\infty})\mathfrak{h}$ is isomorphic  either to the two-element cyclic group $\mathbb{Z}_2$ or to the additive group of integers $\mathbb{Z}(+)$. Also we prove that the monoid  $\mathbf{I}\mathbb{N}_{\infty}$ is not  finitely generated, and, moreover, monoid $\mathbf{I}\mathbb{N}_{\infty}$ does not contain a minimal generating set.
\end{abstract}

\maketitle


\section{\textbf{Introduction and preliminaries}}
In this paper we shall follow the terminology of \cite{Clifford-Preston-1961-1967, Lawson-1998}.
We shall denote the first infinite cardinal by $\omega$ and the cardinality of a set $A$ by $|A|$. For any positive integer $n$ by $\mathscr{S}_n$ we denote the group of permutations of the set $\{1,\ldots,n\}$.

We shall say that a non-empty subset $A$ of  a semigroup $S$ \emph{generates} $S$, or $A$ is a \emph{set of generators} of $S$, or $A$ is a \emph{generating set} of $S$, if for any $s\in S$ there exist $a_1,\ldots,a_k\in A$ such that $s=a_1\cdots a_k$. For any non-empty subset $A$ of a semigroup $S$ by $\langle A\rangle$ we denote a subsemigroup of $S$ which is generated by $A$.  A generating set $A$ of a semigroup $S$ is called \emph{minimal generating}, if $A$ does not properly contain any generating set of $S$. It is obvious that every finite generation set of a semigroup has a minimal generating set.

A semigroup $S$ is called {\it inverse} if for any
element $x\in S$ there exists a unique $x^{-1}\in S$ such that
$xx^{-1}x=x$ and $x^{-1}xx^{-1}=x^{-1}$. The element $x^{-1}$ is
called the {\it inverse of} $x\in S$. If $S$ is an inverse
semigroup, then the function $\operatorname{inv}\colon S\to S$
which assigns to every element $x$ of $S$ its inverse element
$x^{-1}$ is called the {\it inversion}.

If $S$ is a semigroup, then we shall denote the subset of all
idempotents in $S$ by $E(S)$. If $S$ is an inverse semigroup, then
$E(S)$ is closed under multiplication and we shall refer to $E(S)$ as a
\emph{band} (or the \emph{band of} $S$). Then the semigroup
operation on $S$ determines the following partial order $\preccurlyeq$
on $E(S)$: $e\preccurlyeq f$ if and only if $ef=fe=e$. This order is
called the {\em natural partial order} on $E(S)$. A
\emph{semilattice} is a commutative semigroup of idempotents.

If $S$ is an inverse semigroup then the semigroup operation on $S$ determines the following partial order $\preccurlyeq$
on $S$: $s\preccurlyeq t$ if and only if there exists $e\in E(S)$ such that $s=te$. This order is
called the {\em natural partial order} on $S$ \cite{Wagner-1952}.

A congruence $\mathfrak{C}$ on a semigroup $S$ is called
\emph{non-trivial} if $\mathfrak{C}$ is distinct from the universal and
identity congruences on $S$, and a \emph{group congruence} if the quotient
semigroup $S/\mathfrak{C}$ is a group.
Every inverse semigroup $S$ admits the \emph{least} (\emph{minimum}) group
congruence $\mathfrak{C}_{\mathbf{mg}}$:
\begin{equation*}
    a\mathfrak{C}_{\mathbf{mg}} b \; \hbox{ if and only if there exists }\;
    e\in E(S) \; \hbox{ such that }\; ae=be
\end{equation*}
(see \cite[Lemma~III.5.2]{Petrich-1984}).

The bicyclic monoid ${\mathscr{C}}(p,q)$ is the semigroup with the identity $1$ generated by two elements $p$ and $q$ subjected only to the condition $pq=1$. The semigroup operation on ${\mathscr{C}}(p,q)$ is determined as
follows:
\begin{equation*}
    q^kp^l\cdot q^mp^n=q^{k+m-\min\{l,m\}}p^{l+n-\min\{l,m\}}.
\end{equation*}
It is well known that the bicyclic monoid ${\mathscr{C}}(p,q)$ is a bisimple (and hence simple) combinatorial $E$-unitary inverse semigroup and every non-trivial congruence on ${\mathscr{C}}(p,q)$ is a group congruence \cite{Clifford-Preston-1961-1967}.

If $\alpha\colon X\rightharpoonup Y$ is a partial map, then we shall denote
the domain and the range of $\alpha$ by $\operatorname{dom}\alpha$ and $\operatorname{ran}\alpha$, respectively. A partial map $\alpha\colon X\rightharpoonup Y$ is called \emph{cofinite} if both sets $X\setminus\operatorname{dom}\alpha$ and $Y\setminus\operatorname{ran}\alpha$ are finite.

Let $\mathscr{I}_\lambda$ denote the set of all partial one-to-one
transformations of a non-zero  cardinal $\lambda$ together
with the following semigroup operation:
\begin{equation*}
x(\alpha\beta)=(x\alpha)\beta \quad \hbox{if} \quad x\in\operatorname{dom}(\alpha\beta)=\{
y\in\operatorname{dom}\alpha\colon y\alpha\in\operatorname{dom}\beta\}, \qquad  \hbox{for} \quad
\alpha,\beta\in\mathscr{I}_\lambda.
\end{equation*}
 The semigroup
$\mathscr{I}_\lambda$ is called the \emph{symmetric inverse} (\emph{monoid})
\emph{semigroup} over cardinal $\lambda$~(see \cite{Clifford-Preston-1961-1967}). The symmetric inverse
semigroup was introduced by Wagner~\cite{Wagner-1952} and it plays
a major role in the theory of semigroups. By $\mathscr{I}^{\mathrm{cf}}_\lambda$ we denote a
subsemigroup of injective partial selfmaps of $\lambda$ with
cofinite domains and ranges in $\mathscr{I}_\lambda$. Obviously, $\mathscr{I}^{\mathrm{cf}}_\lambda$ is an inverse
submonoid of the semigroup $\mathscr{I}_\lambda$. The
semigroup $\mathscr{I}^{\mathrm{cf}}_\lambda$  is called the \emph{monoid of
injective partial cofinite selfmaps} of $\lambda$ \cite{Gutik-Repovs-2015}.


A partial transformation $\alpha\colon (X,d)\rightharpoonup (X,d)$ of a metric space $(X,d)$ is called \emph{isometric} or a \emph{partial isometry}, if $d(x\alpha,y\alpha)=d(x,y)$ for all $x,y\in \operatorname{dom}\alpha$. It is obvious that the composition of two partial isometries of a metric space $(X,d)$ is a partial isometry, and the converse partial map to a partial isometry is a partial isometry, too. Hence the set of partial isometries of a metric space $(X,d)$ with the operation of composition of partial isometries is an inverse submonoid of the symmetric inverse monoid over the cardinal $|X|$. Also, it is obvious that the set of partial cofinite isometries of a metric space $(X,d)$ with the operation of composition of partial isometries is an inverse submonoid of the monoid of injective partial cofinite selfmaps of the cardinal $|X|$.


The semigroup $\mathbf{ID}_{\infty}$ of all partial cofinite isometries of the set of integers $\mathbb{Z}$ with the usual metric $d(n,m)=|n-m|$, $n,m\in \mathbb{Z}$, in the Bezushchak papers \cite{Bezushchak-2004, Bezushchak-2008} is considered. In \cite{Bezushchak-2004}  the generators of the semigroup $\mathbf{ID}_{\infty}$ are described and it is proved therein that $\mathbf{ID}_{\infty}$ has the exponential growth. We remark that the semigroup $\mathbf{ID}_{\infty}$ is an inverse submonoid of the  monoid of all partial cofinite bijections of $\mathbb{Z}$, and  elements of $\mathbf{ID}_{\infty}$ are restrictions of isometries of $\mathbb{Z}$ onto its cofinite subsets in the Lawson interpretation (see \cite[p.~9]{Lawson-1998}). Green's relations and principal ideals of $\mathbf{ID}_{\infty}$ are described in \cite{Bezushchak-2008}. In \cite{Gutik-Savchuk-2017} it is shown that the quotient semigroup $\mathbf{ID}_{\infty}/\mathfrak{C}_{\mathbf{mg}}$ is isomorphic to the group ${\mathbf{Iso}}(\mathbb{Z})$ of all isometries of $\mathbb{Z}$, the semigroup $\mathbf{ID}_{\infty}$ is $F$-inverse, and $\mathbf{ID}_{\infty}$ is isomorphic to the semidirect product ${\mathbf{Iso}}(\mathbb{Z})\ltimes_\mathfrak{h}\mathscr{P}_{\!\infty}(\mathbb{Z})$ of the free semilattice with identity $(\mathscr{P}_{\!\infty}(\mathbb{Z}),\cup)$ by the group ${\mathbf{Iso}}(\mathbb{Z})$. Also in \cite{Gutik-Savchuk-2017} there are investigated semigroup and shift-continuous topologies on $\mathbf{ID}_{\infty}$ and embedding of the discrete semigroup $\mathbf{ID}_{\infty}$ into compact-like topological semigroups.


Later we assume that on $\mathbb{N}$ and $\mathbb{Z}$ the usual linear order is considered.

Let $\mathbf{I}\mathbb{N}_{\infty}$ be the set of all partial cofinite isometries of the set of positive integers $\mathbb{N}$ with the usual metric $d(n,m)=|n-m|$, $n,m\in \mathbb{N}$. Then $\mathbf{I}\mathbb{N}_{\infty}$ with the operation of composition of partial isometries is an inverse submonoid of $\mathscr{I}_\omega$. The semigroup $\mathbf{I}\mathbb{N}_{\infty}$ of all partial cofinite isometries of positive integers is studied in \cite{Gutik-Savchuk-2018}. There we described the Green relations on the semigroup $\mathbf{I}\mathbb{N}_{\infty}$, its band, and proved that $\mathbf{I}\mathbb{N}_{\infty}$ is a simple $E$-unitary $F$-inverse semigroup. Also in \cite{Gutik-Savchuk-2018}, the least group congruence $\mathfrak{C}_{\mathbf{mg}}$ on $\mathbf{I}\mathbb{N}_{\infty}$ is described and it is proved that the quotient-semigroup  $\mathbf{I}\mathbb{N}_{\infty}/\mathfrak{C}_{\mathbf{mg}}$ is isomorphic to the additive group of integers $\mathbb{Z}(+)$. An example of a non-group congruence on the semigroup $\mathbf{I}\mathbb{N}_{\infty}$ is presented. Also in \cite{Gutik-Savchuk-2018}, we proved that a congruence on the semigroup $\mathbf{I}\mathbb{N}_{\infty}$ is a group congruence if and only if its restriction onto an isomorphic  copy of the bicyclic semigroup in $\mathbf{I}\mathbb{N}_{\infty}$ is a group congruence and it is shown that $\mathbf{I}\mathbb{N}_{\infty}$ has a non-trivial homomorphic retract which is isomorphic to the bicyclic semigroup. Another non-trivial homomorphic retracts of the monoid $\mathbf{I}\mathbb{N}_{\infty}$ is constructed in \cite{Savchuk-2019}.


The semigroup of monotone, non-decreasing, injective partial
transformations $\varphi$ of $\mathbb{N}$ such that the sets
$\mathbb{N}\setminus\operatorname{dom}\varphi$ and
$\mathbb{N}\setminus\operatorname{ran}\varphi$ are finite  was introduced in \cite{Gutik-Repovs-2011} and was denoted by
$\mathscr{I}_{\infty}^{\!\nearrow}(\mathbb{N})$.
Obviously, $\mathscr{I}_{\infty}^{\!\nearrow}(\mathbb{N})$ is an
inverse subsemigroup of the semigroup $\mathscr{I}_\omega$. The
semigroup $\mathscr{I}_{\infty}^{\!\nearrow}(\mathbb{N})$ is
called \emph{the semigroup of cofinite monotone partial
bijections} of $\mathbb{N}$.
In \cite{Gutik-Repovs-2011} Gutik and Repov\v{s} studied properties of the
semigroup $\mathscr{I}_{\infty}^{\!\nearrow}(\mathbb{N})$. In particular, they showed that
$\mathscr{I}_{\infty}^{\!\nearrow}(\mathbb{N})$ is an inverse bisimple semigroup and all of its non-trivial group homomorphisms are either isomorphisms or
group homomorphisms. It is obvious that $\mathbf{I}\mathbb{N}_{\infty}$ is an inverse submonoid of $\mathscr{I}_{\infty}^{\!\nearrow}(\mathbb{N})$.

Doroshenko in \cite{Doroshenko-2005, Doroshenko-2009} studied the semigroups of endomorphisms of linearly ordered sets $\mathbb{N}$ and $\mathbb{Z}$ and their subsemigroups of cofinite endomorphisms $\mathcal{O}_{fin}(\mathbb{N})$ and $\mathcal{O}_{fin}(\mathbb{Z})$. In \cite{Doroshenko-2009} he described Green's relations, groups of automorphisms, conjugacy, centralizers of elements, growth, and free subsemigroups in these semigroups. In particular, in \cite{Doroshenko-2009} it is proved that, in $\mathcal{O}_{fin}(\mathbb{N})$  the  group  of  automorphisms  consists  only  of  the  identity  mapping,  whereas  the  groups  of  automorphisms of $\mathcal{O}_{fin}(\mathbb{Z})$ is isomorphic to the semigroup of integers with operation of addition and consist only of inner automorphisms.
In \cite{Doroshenko-2005} it was shown  that  both these semigroups do not admit an irreducible system of generators. In their  subsemigroups of cofinite functions all irreducible systems of generators are
described here. Also, here the last semigroups are presented in terms of generators and relations.

A partial map $\alpha\colon \mathbb{N}\rightharpoonup \mathbb{N}$ is
called \emph{almost monotone} if there exists a finite subset $A$ of
$\mathbb{N}$ such that the restriction
$\alpha\mid_{\mathbb{N}\setminus A}\colon \mathbb{N}\setminus
A\rightharpoonup \mathbb{N}$ is a monotone partial map.
By $\mathscr{I}_{\infty}^{\,\Rsh\!\!\!\nearrow}(\mathbb{N})$ we
 denote the semigroup of monotone, almost non-decreasing,
injective partial transformations of $\mathbb{N}$ such that the sets
$\mathbb{N}\setminus\operatorname{dom}\varphi$ and
$\mathbb{N}\setminus\operatorname{ran}\varphi$ are finite for all
$\varphi\in\mathscr{I}_{\infty}^{\,\Rsh\!\!\!\nearrow}(\mathbb{N})$.
Obviously, $\mathscr{I}_{\infty}^{\,\Rsh\!\!\!\nearrow}(\mathbb{N})$
is an inverse subsemigroup of the semigroup $\mathscr{I}_\omega$ and
the semigroup $\mathscr{I}_{\infty}^{\!\nearrow}(\mathbb{N})$ is an
inverse subsemigroup of
$\mathscr{I}_{\infty}^{\,\Rsh\!\!\!\nearrow}(\mathbb{N})$ as well. The
semigroup $\mathscr{I}_{\infty}^{\,\Rsh\!\!\!\nearrow}(\mathbb{N})$
is called \emph{the semigroup of cofinite almost monotone partial
bijections} of $\mathbb{N}$.
In  \cite{Chuchman-Gutik-2010} the semigroup
$\mathscr{I}_{\infty}^{\,\Rsh\!\!\!\nearrow}(\mathbb{N})$ is studied. In particular, it was shown that the semigroup
$\mathscr{I}_{\infty}^{\,\Rsh\!\!\!\nearrow}(\mathbb{N})$ is inverse,
bisimple and all of its non-trivial group homomorphisms are either
isomorphisms or group homomorphisms.  In  \cite{Gutik-Savchuk-2019} we showed that every automorphism of a full inverse subsemigroup of $\mathscr{I}_{\infty}^{\!\nearrow}(\mathbb{N})$ which contains the semigroup $\mathscr{C}_{\mathbb{N}}$ is the identity map. Also, here we  constructed a submonoid $\mathbf{I}\mathbb{N}_{\infty}^{[\underline{1}]}$ of $\mathscr{I}_{\infty}^{\,\Rsh\!\!\!\nearrow}(\mathbb{N})$ with the following property: if $S$ be an inverse subsemigroup of $\mathscr{I}_{\infty}^{\,\Rsh\!\!\!\nearrow}(\mathbb{N})$ such that $S$ contains $\mathbf{I}\mathbb{N}_{\infty}^{[\underline{1}]}$ as a submonoid, then every non-identity congruence $\mathfrak{C}$ on $S$ is a group congruence. We show that if $S$ is an inverse submonoid of $\mathscr{I}_{\infty}^{\,\Rsh\!\!\!\nearrow}(\mathbb{N})$ such that $S$ contains $\mathscr{C}_{\mathbb{N}}$ as a submonoid then $S$ is simple and the quotient semigroup $S/\mathfrak{C}_{\mathbf{mg}}$, where $\mathfrak{C}_{\mathbf{mg}}$ is the minimum group congruence on $S$, is isomorphic to the additive group of integers. Also,  topologizations of inverse submonoids of $\mathscr{I}_{\infty}^{\,\Rsh\!\!\!\nearrow}(\mathbb{N})$  and embeddings of  such semigroups into compact-like topological semigroups are given in \cite{Chuchman-Gutik-2010, Gutik-Savchuk-2019}. Similar results for semigroups of cofinite almost monotone partial
bijections and cofinite almost monotone partial bijections of $\mathbb{Z}$ were obtained in \cite{Gutik-Repovs-2012}.

\smallskip

In the present paper we show that the monoid $\mathbf{I}\mathbb{N}_{\infty}$ does not embed isomorphically into the semigroup $\mathbf{ID}_{\infty}$.  Moreover every non-annihilating homomorphism $\mathfrak{h}\colon \mathbf{I}\mathbb{N}_{\infty}\to\mathbf{ID}_{\infty}$ has the following property:   the image $(\mathbf{I}\mathbb{N}_{\infty})\mathfrak{h}$ is isomorphic  either to $\mathbb{Z}_2$ or to $\mathbb{Z}(+)$. Also we prove that the monoid  $\mathbf{I}\mathbb{N}_{\infty}$ does not have a finite set of generators, and moreover monoid $\mathbf{I}\mathbb{N}_{\infty}$ does not contain a minimal generating set.

\section{\textbf{On homomorphisms from $\mathbf{I}\mathbb{N}_{\infty}$ into $\mathbf{ID}_{\infty}$} }\label{section-2}

The definition of the semigroup $\mathbf{ID}_{\infty}$ implies that for any $\alpha\in \mathbf{ID}_{\infty}$ there exists a unique element $\gamma_{\alpha}$ of the group of units of $\mathbf{ID}_{\infty}$ such that $\alpha\preccurlyeq\gamma_{\alpha}$ (see \cite{Gutik-Savchuk-2017}). Also we have that $\left|\mathbb{Z}\setminus\operatorname{dom}\alpha\right|=\left|\mathbb{Z}\setminus\operatorname{ran}\alpha\right|$ for each $\alpha\in\mathbf{ID}_{\infty}$. Hence we get the following obvious lemma:

\begin{lemma}\label{lemma-2.1}
If $\alpha=\beta\gamma$ for some  $\alpha,\beta,\gamma\in\mathbf{ID}_{\infty}$ then
\begin{equation*}
  \max\left\{\left|\mathbb{Z}\setminus\operatorname{dom}\beta\right|, \left|\mathbb{Z}\setminus\operatorname{dom}\gamma\right|\right\}\leq \left|\mathbb{Z}\setminus\operatorname{dom}\alpha\right|\leq \left|\mathbb{Z}\setminus\operatorname{dom}\beta\right|+ \left|\mathbb{Z}\setminus\operatorname{dom}\gamma\right|.
\end{equation*}
\end{lemma}

\begin{proposition}\label{proposition-2.2}
The semigroup $\mathbf{ID}_{\infty}$ does not contain an isomorphic copy of the bicyclic semigroup.
\end{proposition}

\begin{proof}
Suppose to the contrary that there exists a subsemigroup $S$ of $\mathbf{ID}_{\infty}$ which is isomorphic to the bicyclic semigroup $\mathscr{C}(p,q)$. Let $\mathfrak{h}\colon\mathscr{C}(p,q)\to S$ be an embedding isomorphism. Put $(1)\mathfrak{h}=\varepsilon_0$, $(qp)\mathfrak{h}=\varepsilon_1$, $(p)\mathfrak{h}=\alpha$ and $(q)\mathfrak{h}=\beta$. Then $\varepsilon_0$ and $\varepsilon_1$ are idempotent of $\mathbf{ID}_{\infty}$ such that $\varepsilon_1\preccurlyeq\varepsilon_0$. The definition of the semigroup $\mathbf{ID}_{\infty}$ implies that $\varepsilon_0$ and $\varepsilon_1$ are the identity maps of $\operatorname{dom}\varepsilon_0$ and $\operatorname{dom}\varepsilon_1$, respectively, and moreover $\operatorname{dom}\varepsilon_1\varsubsetneq\operatorname{dom}\varepsilon_0$. Since $1=p(qp)p$, we get that $\varepsilon_0=\beta\varepsilon_1\alpha$. The latter equality and Lemma~\ref{lemma-2.1} imply that
\begin{equation*}
  \left|\mathbb{Z}\setminus\operatorname{dom}\varepsilon_1\right|\leq \left|\mathbb{Z}\setminus\operatorname{dom}\varepsilon_0\right|.
\end{equation*}
The obtained inequality contradicts the inclusion $\operatorname{dom}\varepsilon_1\varsubsetneq\operatorname{dom}\varepsilon_0$, because $\varepsilon_0\neq\varepsilon_1$.
\end{proof}

It is obvious that for every $\alpha\in\mathbf{I}\mathbb{N}_{\infty}$ there exist infinitely many $\gamma\in\mathbf{ID}_{\infty}$ such that $\alpha$ is the restriction of $\gamma$ onto $\mathbb{N}$. This motivated Taras Banakh to ask:

\begin{question}\label{question-2.3}
Does the semigroup  $\mathbf{ID}_{\infty}$ contain an isomorphic copy of $\mathbf{I}\mathbb{N}_{\infty}$?
\end{question}

In this section we give a negative answer on this question.

\begin{remark}\label{remark-2.4}
We observe that the bicyclic semigroup is isomorphic to the
semigroup $\mathscr{C}_{\mathbb{N}}$ which is
generated by partial transformations $\alpha$ and $\beta$ of the set
of positive integers $\mathbb{N}$, defined as follows:
\begin{equation*}
\operatorname{dom}\alpha=\mathbb{N}, \qquad \operatorname{ran}\alpha=\mathbb{N}\setminus\{1\},  \qquad (n)\alpha=n+1
\end{equation*}
and
\begin{equation*}
\operatorname{dom}\beta=\mathbb{N}\setminus\{1\}, \qquad \operatorname{ran}\beta=\mathbb{N},  \qquad (n)\beta=n-1
\end{equation*}
(see Exercise~IV.1.11$(ii)$ in \cite{Petrich-1984}). It is obvious that $\mathscr{C}_{\mathbb{N}}$ is a submonoid  of $\mathbf{I}\mathbb{N}_{\infty}$.
\end{remark}

Proposition~\ref{proposition-2.2} and Remark~\ref{remark-2.4} imply the following statement which gives a negative answer to Question~\ref{question-2.3}.

\begin{theorem}\label{theorem-2.5}
The semigroup  $\mathbf{ID}_{\infty}$ does not contain an isomorphic copy of the semigroup $\mathbf{I}\mathbb{N}_{\infty}$.
\end{theorem}

Next we shall discuss maximal subgroups (i.e., on $\mathscr{H}$-classes with an idempotent) in the semigroup $\mathbf{ID}_{\infty}$.

\smallskip

The following statement belongs to the folklore of the geometric group theory.

\begin{lemma}\label{lemma-2.6}
The group of isometries of the set of integers $\mathbb{Z}$ with the usual metric is isomorphic to the semidirect product $\mathbb{Z}(+)\rtimes\mathbb{Z}_2$.
\end{lemma}

The following lemma describes cyclic subgroups of the group of isometries of the set of integers $\mathbb{Z}$ with the usual metric.

\begin{lemma}\label{lemma-2.7}
Let $G$ be a cyclic subgroup of the group of isometries of the set of integers $\mathbb{Z}$ with the usual metric. Then only one of the following conditions holds:
\begin{itemize}
  \item[$(i)$] $G$ is a singleton;
  \item[$(ii)$] $G$ is isomorphic to $\mathbb{Z}_2$;
  \item[$(iii)$] $G$ is isomorphic to $\mathbb{Z}(+)$.
\end{itemize}
\end{lemma}

\begin{proof}
Fix a generator $(a,b)$ of $G$. Next we consider all possible cases.

\smallskip

\textbf{1.} Suppose that $(a,b)=\left(0,\overline{0}\right)$ where $0$ and $\overline{0}$ are neutral elements of $\mathbb{Z}(+)$ and $\mathbb{Z}_2$, respectively. Then the group operation of $\mathbb{Z}(+)\rtimes\mathbb{Z}_2$ implies that $\left(0,\overline{0}\right)^n=\left(0,\overline{0}\right)$ for any integer $n$, and hence $G$ is a singleton.

\smallskip

\textbf{2.} Suppose that $(a,b)=\left(0,\overline{1}\right)$ where $\overline{1}$ is a non-neutral element of $\mathbb{Z}_2$. Then we have that $\left(0,\overline{1}\right)^2=\left(0,\overline{0}\right)$, and hence $G$ is isomorphic to $\mathbb{Z}_2$.

\smallskip

\textbf{3.} Suppose that $(a,b)=\left(g,\overline{0}\right)$ where $g$ is a non-neutral element of $\mathbb{Z}(+)$. Then $\left(g,\overline{0}\right)^n=\left(n\cdot g,\overline{0}\right)$ for any integer $n$, and hence $G$ is isomorphic to $\mathbb{Z}(+)$.

\smallskip

\textbf{4.} Suppose that $(a,b)=\left(g,\overline{1}\right)$ where $g$ is a non-neutral element of $\mathbb{Z}(+)$. Then we have that
\begin{equation*}
  \left(g,\overline{1}\right)\left(g,\overline{1}\right)=\left(g-g,\overline{1}\cdot\overline{1}\right)= \left(0,\overline{0}\right),
\end{equation*}
and hence $G$ is isomorphic to $\mathbb{Z}_2$.
\end{proof}

A subset $C\subseteq \mathbb{R}$ is called \emph{symmetric} in $\mathbb{R}$ if there exists a number $c\in \mathbb{R}$ (the \emph{center} of $C$) such that $c+x\in C$ if and only if $c-x\in C$. A subset $C\subseteq \mathbb{Z}$ is called \emph{symmetric} in $\mathbb{Z}$ if $C$ is symmetric in $\mathbb{R}$.

\begin{remark}\label{remark-2.8}
We observe that a subset $C$ is symmetric in $\mathbb{Z}$ if and only if $\mathbb{Z}\setminus C$ is symmetric in $\mathbb{Z}$. Also, if $\mathbb{Z}$ endowed with the usual metric, then the partial mapping $f_C\colon C\to C$, $c+x\mapsto c-x$ which is determined by the symmetry of the symmetric set $C$ with the centre $c\in\mathbb{R}$ is a partial isometry of $\mathbb{Z}$. In this case we shall say that the \emph{partial map $f_C$ determines a symmetry of $C$}.
\end{remark}

\begin{lemma}\label{lemma-2.9}
Let $C$ be a proper cofinite subset of $\mathbb{Z}$ and $\gamma\colon \mathbb{Z}\rightharpoonup\mathbb{Z}$ be a partial isometry of $\mathbb{Z}$ such that $\operatorname{dom}\gamma=\operatorname{ran}\gamma=C$. Then $\gamma$ is either the identity map of $C$ or $\gamma$ determines a symmetry of $C$.
\end{lemma}

\begin{proof}
Suppose that the partial map $\gamma$ is a nonidentity. Then $\gamma$ is an element of the semigroup $\mathbf{ID}_{\infty}$. By Corollary~1 of \cite{Gutik-Savchuk-2017}, $\mathbf{ID}_{\infty}$  is an $F$-inverse semigroup, and moreover there exists a unique element $\sigma_\gamma$ of the group of units of $\mathbf{ID}_{\infty}$ such that $\gamma\preccurlyeq\sigma_\gamma$. The latter implies that the partial map $\gamma$ extends to the unique isometry $\sigma_\gamma$ of $\mathbb{Z}$. It is obvious that the restriction of $\sigma_\gamma$ onto the set $\mathbb{Z}\setminus C$ is an isometry of $\mathbb{Z}\setminus C$. We denote this isometry by $\gamma^\circ$. Since $\gamma$ is a nonidentity, so is $\gamma^\circ$. Since $C$ is a proper cofinite subset of $\mathbb{Z}$, $(\max(\mathbb{Z}\setminus C))\gamma^\circ=\min(\mathbb{Z}\setminus C)$ and $(\min(\mathbb{Z}\setminus C))\gamma^\circ=\max(\mathbb{Z}\setminus C)$. Then the isometry of $\mathbb{Z}\setminus C$ by $\gamma^\circ$ implies that
\begin{equation*}
c=\frac{\min(\mathbb{Z}\setminus C)+\max(\mathbb{Z}\setminus C)}{2}
\end{equation*}
is the centre of symmetry of $\mathbb{Z}\setminus C$. It is obvious that $c$ is the centre of symmetry of $C$. This implies the statement of the lemma.
\end{proof}

Since any elements $\alpha$ and $\beta$ are $\mathscr{H}$-equivalent in $\mathbf{ID}_{\infty}$ if and only if $\operatorname{dom}\alpha=\operatorname{dom}\beta$ and $\operatorname{ran}\alpha=\operatorname{ran}\beta$, Lemma~\ref{lemma-2.9} implies the following proposition.

\begin{proposition}\label{proposition-2.10}
Every subgroup of $\mathbf{ID}_{\infty}$ distinct from its group of units is either trivial or isomorphic to $\mathbb{Z}_2$.
\end{proposition}

\begin{theorem}\label{theorem-2.11}
Let $S$ be an inverse submonoid of $\mathscr{I}_{\infty}^{\,\Rsh\!\!\!\nearrow}(\mathbb{N})$ which contains $\mathscr{C}_{\mathbb{N}}$ as a submonoid.  Then for any homomorphism $\mathfrak{h}\colon S\to \mathbf{ID}_{\infty}$  one of the following conditions holds:
\begin{itemize}
  \item[$(i)$] the image $(S)\mathfrak{h}$ is a singleton, i.e.,  $\mathfrak{h}$ is an annihilating homomorphism;
  \item[$(ii)$] the image $(S)\mathfrak{h}$ is isomorphic to $\mathbb{Z}_2$;
  \item[$(iii)$] the image $(S)\mathfrak{h}$ is isomorphic to $\mathbb{Z}(+)$.
\end{itemize}
\end{theorem}

\begin{proof}
Suppose that the homomorphism $\mathfrak{h}\colon S\to \mathbf{ID}_{\infty}$ is not annihilating. Since by Remark~\ref{remark-2.4} the monoid $\mathscr{C}_{\mathbb{N}}$ is isomorphic to the bicyclic semigroup, Theorem~\ref{theorem-2.5} implies that the restriction $\mathfrak{h}|_{\mathscr{C}_{\mathbb{N}}}\colon \mathscr{C}_{\mathbb{N}}\to \mathbf{ID}_{\infty}$ is not an injective homomorphism. Then by Corollary~1.32 of \cite{Clifford-Preston-1961-1967} the image $(\mathscr{C}_{\mathbb{N}})\mathfrak{h}$ is a cyclic subgroup of $\mathbf{ID}_{\infty}$ such that $\{(\mathbb{I})\mathfrak{h}\}=(E(\mathscr{C}_{\mathbb{N}}))\mathfrak{h}$.

We shall show that for any idempotent $\varepsilon\in S$ we have that $(\varepsilon)\mathfrak{h}=(\mathbb{I})\mathfrak{h}$. Since $\varepsilon\in \mathscr{I}_{\infty}^{\,\Rsh\!\!\!\nearrow}(\mathbb{N})$, there exists a smallest positive integer $n_\varepsilon$ such that $n\in\operatorname{dom}\varepsilon$ for any $n\geq n_\varepsilon$. Put $\varepsilon_0$ be the identity map of the set $\left\{j\in\mathbb{N}\colon j\geq n_\varepsilon\right\}$. Then $\varepsilon_0$ is an idempotent of $\mathscr{C}_{\mathbb{N}}$ such that $\varepsilon_0\preccurlyeq\varepsilon$ in $S$. The above arguments in the previous paragraph imply that
\begin{equation*}
  (\varepsilon)\mathfrak{h}=(\varepsilon\mathbb{I})\mathfrak{h}=(\varepsilon)\mathfrak{h}(\mathbb{I})\mathfrak{h}= (\varepsilon)\mathfrak{h}(\varepsilon_0)\mathfrak{h}=(\varepsilon \varepsilon_0)\mathfrak{h}=(\varepsilon_0)\mathfrak{h}.
\end{equation*}
Hence we have that $(E(S))\mathfrak{h}=(E(\mathscr{C}_{\mathbb{N}}))\mathfrak{h}$ is a singleton in $\mathbf{ID}_{\infty}$ and moreover the image $(E(S))\mathfrak{h}$ is an idempotent which is the neutral element of the cyclic subgroup $(\mathscr{C}_{\mathbb{N}})\mathfrak{h}$ in  $\mathbf{ID}_{\infty}$. This implies that the image $(S)\mathfrak{h}$ is a subgroup of $\mathbf{ID}_{\infty}$, i.e., the homomorphisms $\mathfrak{h}\colon S\to \mathbf{ID}_{\infty}$ generates a group congruence $\mathfrak{C}_{\mathfrak{h}}$ on the monoid $S$. By Theorem~4 of \cite{Gutik-Savchuk-2019}, the quotient semigroup $S/\mathfrak{C}_{\mathbf{mg}}$, where $\mathfrak{C}_{\mathbf{mg}}$ is the minimum group congruence on $S$, is isomorphic to the additive group of integers $\mathbb{Z}(+)$. This implies that the image $(S)\mathfrak{h}$ is a cyclic subgroup of $\mathbf{ID}_{\infty}$.
Next we apply Lemma~\ref{lemma-2.7} and Proposition~\ref{proposition-2.10}.
\end{proof}

Theorem~\ref{theorem-2.11} implies the following corollary:

\begin{corollary}\label{corollary-2.12}
Let $\mathfrak{h}\colon \mathbf{I}\mathbb{N}_{\infty}\to\mathbf{ID}_{\infty}$ be an arbitrary homomorphism. Then one of the following conditions holds:
\begin{itemize}
  \item[$(i)$] $\mathfrak{h}$ is an annihilating homomorphism;
  \item[$(ii)$] the image $(\mathbf{I}\mathbb{N}_{\infty})\mathfrak{h}$ is isomorphic to $\mathbb{Z}_2$;
  \item[$(iii)$] the image $(\mathbf{I}\mathbb{N}_{\infty})\mathfrak{h}$ is isomorphic to $\mathbb{Z}(+)$.
\end{itemize}
\end{corollary}

The following example shows that every cofinite (almost) monotone partial bijection of $\mathbb{N}$ extends to a cofinite (almost) monotone partial bijection of $\mathbb{Z}$.

\begin{example}\label{example-2.13}
Fix an arbitrary $\alpha\in\mathscr{I}_{\infty}^{\,\Rsh\!\!\!\nearrow}(\mathbb{N})$  and any non-positive integer $n$. We define a partial map $\alpha_\mathbb{Z}\colon \mathbb{Z}\rightharpoonup\mathbb{Z}$ in the following way. Put
\begin{align*}
  \operatorname{dom}\alpha_\mathbb{Z}=&\operatorname{dom}\alpha\cup\{i\in \mathbb{Z}\colon i\leq n\}, \\
  \operatorname{ran}\alpha_\mathbb{Z}=&\operatorname{ran}\alpha\cup\{i\in \mathbb{Z}\colon i\leq n\}
\end{align*}
and
\begin{equation*}
  (k)\alpha_\mathbb{Z}=
\left\{
  \begin{array}{cl}
    (k)\alpha, & \hbox{if~} k\in \operatorname{dom}\alpha;\\
    k, & \hbox{if~} k\leq n.
  \end{array}
\right.
\end{equation*}
This determines a map $\mathfrak{i}_n\colon \mathscr{I}_{\infty}^{\,\Rsh\!\!\!\nearrow}(\mathbb{N})\to \mathscr{I}^{\looparrowright}_{\infty}(\mathbb{Z})$, where $\mathscr{I}^{\looparrowright}_{\infty}(\mathbb{Z})$ is a monoid of cofinite almost  monotone partial bijection of $\mathbb{Z}$ (see \cite{Gutik-Repovs-2012}). It is obvious that the so defined map $\mathfrak{i}_n\colon \mathscr{I}_{\infty}^{\,\Rsh\!\!\!\nearrow}(\mathbb{N})\to \mathscr{I}^{\looparrowright}_{\infty}(\mathbb{Z})$ is a homomorphism, and moreover in the case $n=0$ the map $\mathfrak{i}_0$ is a monoid homomorphism. Also, if $\alpha$ is an element of the semigroup $\mathscr{I}_{\infty}^{\!\nearrow}(\mathbb{N})$   of cofinite monotone partial
bijections of $\mathbb{N}$, then the above defined extension $\alpha_\mathbb{Z}\colon \mathbb{Z}\rightharpoonup\mathbb{Z}$ of $\alpha$ is a cofinite monotone partial bijection of $\mathbb{Z}$, and hence  $\alpha_\mathbb{Z}\in \mathscr{I}^{\!\nearrow}_{\infty}(\mathbb{Z})$, where $\mathscr{I}^{\!\nearrow}_{\infty}(\mathbb{Z})$ is a monoid of cofinite  monotone partial bijections of $\mathbb{Z}$ (see \cite{Gutik-Repovs-2012}).
\end{example}

\section{\textbf{On generators of the monoid $\mathbf{I}\mathbb{N}_{\infty}$}}\label{section-3}
In \cite{Bezushchak-2004} it is proved that the semigroup $\mathbf{ID}_{\infty}$ is finitely generated and moreover $\mathbf{ID}_{\infty}$ has three generators.
Taras Banakh posed the following question.

\begin{question}\label{question-3.0}
Is the monoid $\mathbf{I}\mathbb{N}_{\infty}$ finitely generated?
\end{question}

In this section we give a negative answer on this question.

\begin{lemma}\label{lemma-3.1}
If $A$  is a set of generators of the monoid $\mathbf{I}\mathbb{N}_{\infty}$, then $A$ contains at least two distinct elements of $\mathscr{C}_{\mathbb{N}}$.
\end{lemma}

\begin{proof}
Let $\alpha$ and $\beta$ be elements of a monoid $\mathscr{C}_{\mathbb{N}}$ which are defined in Remark~\ref{remark-2.4}. Then there exist finitely many $\alpha_1,\ldots,\alpha_k\in A$ such that $\alpha=\alpha_1\ldots\alpha_k$ and $\alpha_1\neq\mathbb{I}$. Since $\operatorname{dom}\alpha=\mathbb{N}$, the definition of the composition of partial maps implies that $\operatorname{dom}\alpha\subseteq \operatorname{dom}\alpha_1$. By Lemma~1 of \cite{Gutik-Savchuk-2018}, every element of $\mathbf{I}\mathbb{N}_{\infty}$ is a partial shift of $\mathbb{N}$, and hence we get that $\operatorname{dom}\alpha_1=\mathbb{N}$. By Lemma~1 of \cite{Gutik-Savchuk-2018} and Remark~\ref{remark-2.4}, we have that $\alpha_1\in \mathscr{C}_{\mathbb{N}}$. If $\beta=\beta_1\ldots\beta_j$ for some $\beta_1,\ldots,\beta_j\in A$ and $\beta_j\neq\mathbb{I}$, then dually we get that $\beta_j\in\mathscr{C}_{\mathbb{N}}$ with $\operatorname{ran}\beta_j=\mathbb{N}$. This implies the statement of the lemma.
\end{proof}

\begin{remark}
We observe that the set $A_0=\{\alpha,\beta\}$  is not a unique set of generators of the monoid $\mathscr{C}_{\mathbb{N}}$. It is obvious that for any positive integer $n\geqslant 2$ any of the following sets $A_n=\{\alpha^n,\beta\}$ and $B_n=\{\alpha,\beta^n\}$ generates the monoid $\mathscr{C}_{\mathbb{N}}$.
\end{remark}

Next we  need some notions defined in \cite{Gutik-Savchuk-2018} and \cite{Gutik-Savchuk-2019}.
For an arbitrary positive integer $n_0$ we denote
\begin{equation*}
[n_0)=\left\{n\in\mathbb{N}\colon n\geqslant n_0\right\}.
\end{equation*}
Since the set of all positive integers is well ordered, the definition of the semigroup $\mathscr{I}_{\infty}^{\,\Rsh\!\!\!\nearrow}(\mathbb{N})$ implies that for every $\gamma\in\mathscr{I}_{\infty}^{\,\Rsh\!\!\!\nearrow}(\mathbb{N})$ there exists the smallest positive integer $n_{\gamma}^{\mathbf{d}}\in\operatorname{dom}\gamma$ such that the restriction $\gamma|_{\left[n_{\gamma}^{\mathbf{d}}\right)}$ of the partial map $\gamma\colon \mathbb{N}\rightharpoonup \mathbb{N}$ onto the set $\left[n_{\gamma}^{\mathbf{d}}\right)$ is an element of the semigroup $\mathscr{C}_{\mathbb{N}}$, i.e., $\gamma|_{\left[n_{\gamma}^{\mathbf{d}}\right)}$ is a some shift of $\left[n_{\gamma}^{\mathbf{d}}\right)$. For every $\gamma\in\mathscr{I}_{\infty}^{\,\Rsh\!\!\!\nearrow}(\mathbb{N})$ we put $\overrightarrow{\gamma}=\gamma|_{\left[n_{\gamma}^{\mathbf{d}}\right)}$, i.e.
\begin{equation*}
\operatorname{dom}\overrightarrow{\gamma}=\left[n_{\gamma}^{\mathbf{d}}\right), \qquad (x)\overrightarrow{\gamma}=(x)\gamma \quad \hbox{for all} \quad x\in \operatorname{dom}\overrightarrow{\gamma} \qquad \hbox{and} \qquad \operatorname{ran}\overrightarrow{\gamma}=\left(\operatorname{dom}\overrightarrow{\gamma}\right)\gamma.
\end{equation*}
Also, we put
\begin{equation*}
\underline{n}_{\gamma}^{\mathbf{d}}=\min\operatorname{dom}\gamma \qquad \hbox{for} \quad \gamma\in\mathscr{I}_{\infty}^{\,\Rsh\!\!\!\nearrow}(\mathbb{N}),
\end{equation*}
It is obvious that $\underline{n}_{\gamma}^{\mathbf{d}}= n_{\gamma}^{\mathbf{d}}$ when $\gamma\in\mathscr{C}_{\mathbb{N}}$ and $\underline{n}_{\gamma}^{\mathbf{d}}< n_{\gamma}^{\mathbf{d}}$ when $\gamma\in\mathscr{I}_{\infty}^{\,\Rsh\!\!\!\nearrow}(\mathbb{N})\setminus\mathscr{C}_{\mathbb{N}}$. Also for any $\gamma\in\mathbf{I}\mathbb{N}_{\infty}$ we denote
\begin{equation*}
  \underline{n}_{\gamma}^{\mathbf{r}}=(\underline{n}_{\gamma}^{\mathbf{d}})\gamma \qquad \hbox{and} \qquad n_{\gamma}^{\mathbf{r}}=(n_{\gamma}^{\mathbf{d}})\gamma.
\end{equation*}

By Lemma~1 of \cite{Gutik-Savchuk-2018} every element of the monoid $\mathbf{I}\mathbb{N}_{\infty}$ is a partial shift of the set $\mathbb{Z}$. This implies the following lemma.

\begin{lemma}\label{lemma-3.3}
For every element $\gamma$ of the monoid $\mathbf{I}\mathbb{N}_{\infty}$ the following equality holds:
\begin{equation*}
  n_{\gamma}^{\mathbf{r}}-\underline{n}_{\gamma}^{\mathbf{r}}=n_{\gamma}^{\mathbf{d}}-\underline{n}_{\gamma}^{\mathbf{d}}.
\end{equation*}
\end{lemma}

\begin{lemma}\label{lemma-3.4}
Let be $\gamma\in \mathscr{C}_{\mathbb{N}}$ and $\delta\in\mathbf{I}\mathbb{N}_{\infty}$. Then
\begin{equation*}
  n_{\gamma\delta}^{\mathbf{d}}-\underline{n}_{\gamma\delta}^{\mathbf{d}}\leqslant n_{\delta}^{\mathbf{d}}-\underline{n}_{\delta}^{\mathbf{d}}.
\end{equation*}
\end{lemma}

\begin{proof}
If $\delta\in \mathscr{C}_{\mathbb{N}}$, then $\gamma\delta\in \mathscr{C}_{\mathbb{N}}$ and hence we have that
$
  n_{\gamma\delta}^{\mathbf{d}}=\underline{n}_{\gamma\delta}^{\mathbf{d}}
$
which implies that
\begin{equation*}
n_{\gamma\delta}^{\mathbf{d}}-\underline{n}_{\gamma\delta}^{\mathbf{d}}= n_{\delta}^{\mathbf{d}}-\underline{n}_{\delta}^{\mathbf{d}}=0.
\end{equation*}

Next we assume that $\delta,\gamma\delta\in \mathbf{I}\mathbb{N}_{\infty}\setminus \mathscr{C}_{\mathbb{N}}$, because in the case when $\gamma\delta\in \mathscr{C}_{\mathbb{N}}$ the above argument implies the require inequality.  Since $\gamma\delta\in \mathbf{I}\mathbb{N}_{\infty}\setminus \mathscr{C}_{\mathbb{N}}$, we get that $n_{\gamma}^{\mathbf{r}}<n_{\delta}^{\mathbf{d}}-1$. It is obvious that if $n_{\gamma}^{\mathbf{r}}\leqslant \underline{n}_{\delta}^{\mathbf{d}}$ then $n_{\gamma\delta}^{\mathbf{r}}=n_{\gamma}^{\mathbf{r}}$ and $\underline{n}_{\gamma\delta}^{\mathbf{r}}=\underline{n}_{\delta}^{\mathbf{r}}$. If $\underline{n}_{\delta}^{\mathbf{d}}<n_{\gamma}^{\mathbf{r}}<n_{\delta}^{\mathbf{d}}-1$ then $n_{\gamma\delta}^{\mathbf{r}}=n_{\gamma}^{\mathbf{r}}$ and $\underline{n}_{\gamma\delta}^{\mathbf{r}}\geqslant\underline{n}_{\delta}^{\mathbf{r}}$. By Lemma~\ref{lemma-3.3} in the both above cases we have that $n_{\gamma\delta}^{\mathbf{d}}-\underline{n}_{\gamma\delta}^{\mathbf{d}}\leqslant n_{\delta}^{\mathbf{d}}-\underline{n}_{\delta}^{\mathbf{d}}.$
\end{proof}

\begin{lemma}\label{lemma-3.5}
Let be $\gamma\in \mathscr{C}_{\mathbb{N}}$ and $\delta\in\mathbf{I}\mathbb{N}_{\infty}$. Then
\begin{equation*}
  n_{\delta\gamma}^{\mathbf{d}}-\underline{n}_{\delta\gamma}^{\mathbf{d}}\leqslant n_{\delta}^{\mathbf{d}}-\underline{n}_{\delta}^{\mathbf{d}}.
\end{equation*}
\end{lemma}

\begin{proof}
By the first paragraph of the proof of Lemma~\ref{lemma-3.3} without loss of generality we may assume that $\delta,\delta\gamma\in \mathbf{I}\mathbb{N}_{\infty}\setminus \mathscr{C}_{\mathbb{N}}$. Since $\delta\gamma\in \mathbf{I}\mathbb{N}_{\infty}\setminus \mathscr{C}_{\mathbb{N}}$, we have that $n_{\gamma}^{\mathbf{d}}<n_{\delta}^{\mathbf{r}}-1$. It is obvious that if $n_{\gamma}^{\mathbf{d}}\leqslant\underline{n}_{\delta}^{\mathbf{r}}$ then $n_{\delta\gamma}^{\mathbf{d}}=n_{\gamma}^{\mathbf{d}}$ and $\underline{n}_{\delta\gamma}^{\mathbf{d}}=\underline{n}_{\delta}^{\mathbf{d}}$. If $\underline{n}_{\delta}^{\mathbf{r}}<n_{\gamma}^{\mathbf{d}}<n_{\delta}^{\mathbf{r}}-1$ then there exists a positive integer $i^\circ\in\operatorname{dom}\delta$ such that $(i^\circ)\delta\geqslant n_{\gamma}^{\mathbf{d}}$ and $(i^\circ)\delta\gamma=\underline{n}_{\delta\gamma}^{\mathbf{r}}$. Hence in this case we have that
$
n_{\delta\gamma}^{\mathbf{d}}-\underline{n}_{\delta\gamma}^{\mathbf{d}}\leqslant n_{\delta\gamma}^{\mathbf{d}}-i^\circ< n_{\delta}^{\mathbf{d}}-\underline{n}_{\delta}^{\mathbf{d}}.
$
\end{proof}

\begin{lemma}\label{lemma-3.6}
Let $k$ be a positive integer $\geqslant 2$. If $\gamma,\delta\in\mathbf{I}\mathbb{N}_{\infty}\setminus\mathscr{C}_{\mathbb{N}}$ such that $\gamma\delta\in\mathbf{I}\mathbb{N}_{\infty}\setminus\mathscr{C}_{\mathbb{N}}$, $n_{\gamma}^{\mathbf{d}}-\underline{n}_{\gamma}^{\mathbf{d}}\leqslant k$ and $n_{\delta}^{\mathbf{d}}-\underline{n}_{\delta}^{\mathbf{d}}\leqslant k$, then
\begin{equation*}
  n_{\gamma\delta}^{\mathbf{d}}-\underline{n}_{\gamma\delta}^{\mathbf{d}}\leqslant k.
\end{equation*}
\end{lemma}

\begin{proof}
We consider all possible cases.

\textbf{1}. If $\underline{n}_{\gamma}^{\mathbf{r}}\leqslant\underline{n}_{\delta}^{\mathbf{d}}$ and $n_{\gamma}^{\mathbf{r}}\leqslant n_{\delta}^{\mathbf{d}}$, then $\underline{n}_{\delta}^{\mathbf{r}}\leqslant\underline{n}_{\gamma\delta}^{\mathbf{r}}<n_{\delta}^{\mathbf{r}}-1$ and $n_{\delta}^{\mathbf{r}}=n_{\gamma\delta}^{\mathbf{r}}$. Hence in this case by Lemma~1 of \cite{Gutik-Savchuk-2018} and Lemma~\ref{lemma-3.3} we have that
\begin{equation*}
  n_{\gamma\delta}^{\mathbf{d}}-\underline{n}_{\gamma\delta}^{\mathbf{d}}=n_{\gamma\delta}^{\mathbf{r}}-\underline{n}_{\gamma\delta}^{\mathbf{r}} \leqslant n_{\delta}^{\mathbf{r}}-\underline{n}_{\delta}^{\mathbf{r}}=n_{\delta}^{\mathbf{d}}-\underline{n}_{\delta}^{\mathbf{d}}  \leqslant k.
\end{equation*}

\textbf{2}. If $\underline{n}_{\gamma}^{\mathbf{r}}>\underline{n}_{\delta}^{\mathbf{d}}$ and $n_{\gamma}^{\mathbf{r}}\leqslant n_{\delta}^{\mathbf{d}}$, then $n_{\delta}^{\mathbf{r}}=n_{\gamma\delta}^{\mathbf{r}}$ and there exists a positive integer $i^\circ\in\operatorname{dom}\gamma$ such that $(i^\circ)\gamma>\underline{n}_{\delta}^{\mathbf{d}}$ and $(i^\circ)\gamma\delta=\underline{n}_{\gamma\delta}^{\mathbf{r}}$. In this case by Lemma~1 of \cite{Gutik-Savchuk-2018} and Lemma~\ref{lemma-3.3} we have that
\begin{equation*}
  n_{\gamma\delta}^{\mathbf{d}}-\underline{n}_{\gamma\delta}^{\mathbf{d}}=n_{\gamma\delta}^{\mathbf{r}}-\underline{n}_{\gamma\delta}^{\mathbf{r}} = n_{\delta}^{\mathbf{r}}-(i^\circ)\gamma\delta<
  n_{\delta}^{\mathbf{r}}-\underline{n}_{\delta}^{\mathbf{r}}=
  n_{\delta}^{\mathbf{d}}-\underline{n}_{\delta}^{\mathbf{d}} \leqslant k.
\end{equation*}

\textbf{3}. If $\underline{n}_{\gamma}^{\mathbf{r}}\leqslant\underline{n}_{\delta}^{\mathbf{d}}$ and $n_{\gamma}^{\mathbf{r}}> n_{\delta}^{\mathbf{d}}$, then $n_{\gamma\delta}^{\mathbf{r}}=(n_{\gamma}^{\mathbf{r}})\delta$ and there exists a positive integer $j^\circ\in\operatorname{ran}\gamma\cap\operatorname{dom}\delta$ such that $j^\circ\geqslant \underline{n}_{\delta}^{\mathbf{d}}$ and $(j^\circ)\delta=\underline{n}_{\gamma\delta}^{\mathbf{r}}$. In this case by Lemma~1 of \cite{Gutik-Savchuk-2018} and Lemma~\ref{lemma-3.3} we have that
\begin{equation*}
  n_{\gamma\delta}^{\mathbf{d}}-\underline{n}_{\gamma\delta}^{\mathbf{d}}=
  n_{\gamma\delta}^{\mathbf{r}}-\underline{n}_{\gamma\delta}^{\mathbf{r}} = (n_{\gamma}^{\mathbf{r}})\delta-(j^\circ)\delta=n_{\gamma}^{\mathbf{r}}-j^\circ\leqslant n_{\gamma}^{\mathbf{r}}-\underline{n}_{\gamma}^{\mathbf{r}}=
  n_{\gamma}^{\mathbf{d}}-\underline{n}_{\gamma}^{\mathbf{d}}\leqslant k.
\end{equation*}

\textbf{4}. If $\underline{n}_{\gamma}^{\mathbf{r}}>\underline{n}_{\delta}^{\mathbf{d}}$ and $n_{\gamma}^{\mathbf{r}}> n_{\delta}^{\mathbf{d}}$, then $n_{\gamma\delta}^{\mathbf{r}}=(n_{\gamma}^{\mathbf{r}})\delta$ and there exists a positive integer $l^\circ\in\operatorname{ran}\gamma\cap\operatorname{dom}\delta$ such that $l^\circ\geqslant \underline{n}_{\gamma}^{\mathbf{r}}$ and $(l^\circ)\delta=\underline{n}_{\gamma\gamma}^{\mathbf{r}}$. Hence in this case by Lemma~1 of \cite{Gutik-Savchuk-2018} and Lemma~\ref{lemma-3.3} we have that
\begin{equation*}
  n_{\gamma\delta}^{\mathbf{d}}-\underline{n}_{\gamma\delta}^{\mathbf{d}}=
  n_{\gamma\delta}^{\mathbf{r}}-\underline{n}_{\gamma\delta}^{\mathbf{r}} =
  (n_{\gamma}^{\mathbf{r}})\delta-(l^\circ)\delta=
  n_{\gamma}^{\mathbf{r}}-l^\circ\leqslant
  n_{\gamma}^{\mathbf{r}}-\underline{n}_{\gamma}^{\mathbf{r}}=
  n_{\gamma}^{\mathbf{d}}-\underline{n}_{\gamma}^{\mathbf{d}}\leqslant k.
\end{equation*}
This completes the proof of the lemma.
\end{proof}

\begin{theorem}\label{theorem-3.7}
The monoid  $\mathbf{I}\mathbb{N}_{\infty}$ is not  finitely generated.
\end{theorem}

\begin{proof}
Suppose to the contrary that there exists a finite set $A=\left\{\gamma_1,\ldots,\gamma_p\right\}$ of generators of $\mathbf{I}\mathbb{N}_{\infty}$. Lemma~\ref{lemma-3.1} implies that $p\geqslant 3$ and without loss of generality we may assume that $\gamma_1,\gamma_2\in\mathscr{C}_{\mathbb{N}}$ and $\gamma_3,\ldots,\gamma_p\in\mathbf{I}\mathbb{N}_{\infty}\setminus\mathscr{C}_{\mathbb{N}}$. Since the set $A\setminus\left\{\gamma_1,\gamma_2\right\}=\left\{\gamma_3,\ldots,\gamma_p\right\}$ is finite and $\gamma_3,\ldots,\gamma_p\in\mathbf{I}\mathbb{N}_{\infty}\setminus\mathscr{C}_{\mathbb{N}}$, there exists a positive integer $k\geqslant 2$ such that $n_{\gamma_j}^{\mathbf{d}}-\underline{n}_{\gamma_j}^{\mathbf{d}}\leqslant k$ for any $j=3,\ldots,p$.

Since $A$ generates the monoid $\mathbf{I}\mathbb{N}_{\infty}$, Lemmas~\ref{lemma-3.3}, \ref{lemma-3.4}, \ref{lemma-3.5}, and \ref{lemma-3.6} imply that $n_{\gamma}^{\mathbf{d}}-\underline{n}_{\gamma}^{\mathbf{d}}\leqslant k$  for any $\gamma\in \mathbf{I}\mathbb{N}_{\infty}$. Let $\varepsilon^*$ be the identity map of the set $\{1\}\cup\{s\in\mathbb{N}\colon s\geqslant k+2\}$. It is obvious that
\begin{equation*}
  n_{\varepsilon^*}^{\mathbf{d}}-\underline{n}_{\varepsilon^*}^{\mathbf{d}}=k+2-1=k+1,
\end{equation*}
which contradicts the above part of the proof. The obtained contradiction implies the statement of the theorem.
\end{proof}

In the following example we construct a set of generators of the monoid $\mathbf{I}\mathbb{N}_{\infty}$.

\begin{example}\label{example-3.8}
Let $\alpha$ and $\beta$ be elements of the submonoid $\mathscr{C}_{\mathbb{N}}$ of $\mathbf{I}\mathbb{N}_{\infty}$ which are  described in Remark~\ref{remark-2.4}. For every positive integer $k\geqslant 2$ we put $\varepsilon^{[k]}$ to be the identity map of the set $\mathbb{N}\setminus \{k\}$. It is obvious that $\varepsilon^{[k]}$ is an idempotent of $\mathbf{I}\mathbb{N}_{\infty}$ and $\varepsilon^{[k]}\notin \mathscr{C}_{\mathbb{N}}$ for all positive integers $k\geqslant 2$. We claim that the set
\begin{equation*}
  A=\left\{\alpha,\beta\right\}\cup \left\{\varepsilon^{[k]}\colon k\in\mathbb{N}\setminus\{1\}\right\}
\end{equation*}
generates the monoid $\mathbf{I}\mathbb{N}_{\infty}$. Indeed, fix an arbitrary $\gamma\in \mathbf{I}\mathbb{N}_{\infty}$. By Lemma~1 from \cite{Gutik-Savchuk-2018}, $\gamma$ is a partial shift of the set of  integers $\mathbb{Z}$ and hence by Remark~\ref{remark-2.4} there exist a non-negative integers $i$ and $j$ such that $(x)\beta^i\alpha^j=(x)\gamma$ for any $x\in\operatorname{dom}\gamma$ and $\underline{n}_{\gamma}^{\mathbf{d}}$ is the smallest element of $\operatorname{dom}(\beta^i\alpha^j)$. If $\gamma=\beta^i\alpha^j$ then the proof is complete. In the other case we have that $\operatorname{dom}(\beta^i\alpha^j)\setminus\operatorname{dom}\gamma\neq\varnothing$ and put
\begin{equation*}
  \left\{i_1,\ldots,i_p\right\}=\operatorname{dom}(\beta^i\alpha^j)\setminus\operatorname{dom}\gamma.
\end{equation*}
Then Lemma~1 from \cite{Gutik-Savchuk-2018} implies that $\gamma=\varepsilon^{[i_1]}\cdots \varepsilon^{[i_p]}\beta^i\alpha^j$, which implies that the set $A$ generates the monoid $\mathbf{I}\mathbb{N}_{\infty}$.
\end{example}

\begin{remark}\label{remark-3.9}
We observe that for any positive integers $k$ and $l$ such that $k>l\geqslant 2$ we have that
\begin{equation*}
  \varepsilon^{[l]}=\alpha^{k-l}\varepsilon^{[k]}\beta^{k-l}.
\end{equation*}
This implies that the set $A$ from Example~\ref{example-3.8} has not a minimal set of generators of the monoid $\mathbf{I}\mathbb{N}_{\infty}$.
\end{remark}

Example~\ref{example-3.8} and Remark~\ref{remark-3.9} imply the following corollary.

\begin{corollary}\label{corollary-3.10}
Every finitely generated subsemigroup of $\mathbf{I}\mathbb{N}_{\infty}$ is a subsemigroup of an inverse subsemigroup of $\mathbf{I}\mathbb{N}_{\infty}$ generated by three elements.
\end{corollary}

\begin{lemma}\label{lemma-3.11}
Let $A$ be any generating set of the monoid $\mathbf{I}\mathbb{N}_{\infty}$. The there exists a minimal finite subset $A_\mathscr{C}^\circ$ of $A$ such that $\mathscr{C}_{\mathbb{N}}\subseteq \left\langle A_\mathscr{C}^\circ\right\rangle$.
\end{lemma}

\begin{proof}
Let $\alpha$ and $\beta$ be elements of the submonoid $\mathscr{C}_{\mathbb{N}}$ of $\mathbf{I}\mathbb{N}_{\infty}$ which are described in Remark~\ref{remark-2.4}. Then there exist finitely many $\gamma_1,\ldots,\gamma_k,\delta_1,\ldots,\delta_l\in A$ such that $\alpha=\gamma_1\cdots\gamma_k$ and $\beta=\delta_1\cdots\delta_l$. Since $\alpha$ and $\beta$ generate $\mathscr{C}_{\mathbb{N}}$, we obtain that $\left\langle\gamma_1,\ldots,\gamma_k,\delta_1,\ldots,\delta_l\right\rangle\supseteq \mathscr{C}_{\mathbb{N}}$. Since the set $\left\{\gamma_1,\ldots,\gamma_k,\delta_1,\ldots,\delta_l\right\}$ is finite, it contains a minimal subset $A_\mathscr{C}^\circ$ such that $\mathscr{C}_{\mathbb{N}}\subseteq \left\langle A_\mathscr{C}^\circ\right\rangle$.
\end{proof}

For any  integer $j\geqslant 0$ we define
\begin{equation*}
  \mathbf{I}\mathbb{N}_{\infty}^{\boldsymbol{g}[j]}=\left\{\gamma\in \mathbf{I}\mathbb{N}_{\infty}\colon n_{\gamma}^{\mathbf{d}}-\underline{n}_{\gamma}^{\mathbf{d}}\leqslant j\right\}.
\end{equation*}
Therefore, by Lemmas~\ref{lemma-3.4},~\ref{lemma-3.5}, and~\ref{lemma-3.6} we obtain an infinite inverse semigroup series in the monoid $\mathbf{I}\mathbb{N}_{\infty}$:
\begin{equation*}
\mathscr{C}_{\mathbb{N}}=\mathbf{I}\mathbb{N}_{\infty}^{\boldsymbol{g}[0]}= \mathbf{I}\mathbb{N}_{\infty}^{\boldsymbol{g}[1]}\subsetneqq \mathbf{I}\mathbb{N}_{\infty}^{\boldsymbol{g}[2]}\subsetneqq \mathbf{I}\mathbb{N}_{\infty}^{\boldsymbol{g}[1]}\subsetneqq \cdots \subsetneqq \mathbf{I}\mathbb{N}_{\infty}^{\boldsymbol{g}[k]}\subsetneqq \cdots \subset \mathbf{I}\mathbb{N}_{\infty}.
\end{equation*}

Theorems 1, 4, and 5 from \cite{Gutik-Savchuk-2019} imply the following proposition.

\begin{proposition}
For any integer $k\geqslant 0$ the following assertions hold:
\begin{itemize}
  \item[$(i)$] every automorphism of $\mathbf{I}\mathbb{N}_{\infty}^{\boldsymbol{g}[k]}$ is the identity map;
  \item[$(ii)$] the quotient semigroup $\mathbf{I}\mathbb{N}_{\infty}^{\boldsymbol{g}[k]}/\mathfrak{C}_{\mathbf{mg}}$ is isomorphic to the additive group of integers $\mathbb{Z}(+)$;
  \item[$(iii)$] $\mathbf{I}\mathbb{N}_{\infty}^{\boldsymbol{g}[k]}$ is an inverse simple semigroup.
\end{itemize}
\end{proposition}

In the sequel, for any positive integer $j\geqslant 2$ by $\varepsilon^{[j]}$ we shall denote the idempotent which is defined in Example~\ref{example-3.8}.

\begin{lemma}\label{lemma-3.12}
Let $k$ be any integer $\geqslant 2$. If $A$ is a subset of $\mathbf{I}\mathbb{N}_{\infty}$ such that $\mathscr{C}_{\mathbb{N}}$ is a subsemigroup of $\left\langle A\right\rangle$ and $\varepsilon^{[k]}\in\left\langle A\right\rangle$, then $\mathbf{I}\mathbb{N}_{\infty}^{\boldsymbol{g}[k]}$ is a subsemigroup of $\left\langle A\right\rangle$.
\end{lemma}

\begin{proof}
By Remark~\ref{remark-3.9} any idempotent $\varepsilon^{[l]}$ of $\mathbf{I}\mathbb{N}_{\infty}$ such that $l< k$ is generated by the idempotent $\varepsilon^{[k]}$ and the elements $\alpha$ and $\beta$ of $\mathscr{C}_{\mathbb{N}}$. Since $\varepsilon=\varepsilon^{[i_1]}\cdots\varepsilon^{[i_p]}$, where $i_1,\ldots,l_p\leqslant k$, for any idempotent $\varepsilon\in\mathbf{I}\mathbb{N}_{\infty}$ with $\varepsilon\preccurlyeq\beta^k\alpha^k$, we conclude that every idempotent $\varepsilon\preccurlyeq\beta^k\alpha^k$ of $\mathbf{I}\mathbb{N}_{\infty}$ is generated by the set $A$.

Fix any element $\gamma$ of the semigroup $\mathbf{I}\mathbb{N}_{\infty}^{\boldsymbol{g}[k]}$. Then the arguments presented in Example~\ref{example-3.8} show that the partial map $\gamma$ is a partial shift of the set $\operatorname{dom}\gamma$ such that $\gamma$ is the restriction of $\beta^{\underline{n}_{\gamma}^{\mathbf{d}}}\alpha^{\underline{n}_{\gamma}^{\mathbf{r}}}$ onto the set $\operatorname{dom}\gamma$. Since $\alpha^{\underline{n}_{\gamma}^{\mathbf{d}}}\beta^{\underline{n}_{\gamma}^{\mathbf{d}}} \alpha^{\underline{n}_{\gamma}^{\mathbf{r}}}\beta^{\underline{n}_{\gamma}^{\mathbf{r}}}$ is the identity map of $\mathbb{N}$, the previous arguments imply that $\varepsilon_0=\alpha^{\underline{n}_{\gamma}^{\mathbf{d}}}\gamma\beta^{\underline{n}_{\gamma}^{\mathbf{r}}}$ is an idempotent of the monoid $\mathbf{I}\mathbb{N}_{\infty}$. By Lemmas~\ref{lemma-3.4},~\ref{lemma-3.5}, \ref{lemma-3.6} and Lemma~1 of \cite{Gutik-Savchuk-2018}, $\varepsilon_0$ belongs to the semigroup $\mathbf{I}\mathbb{N}_{\infty}^{\boldsymbol{g}[k]}$. By the previous part of the proof there exist $\gamma_1,\ldots,\gamma_n\in A$ such that $\varepsilon_0=\gamma_1\cdots\gamma_n$. Again, since $\gamma$ is the restriction of $\beta^{\underline{n}_{\gamma}^{\mathbf{d}}}\alpha^{\underline{n}_{\gamma}^{\mathbf{r}}}$ onto the set $\operatorname{dom}\gamma$, we obtain that \begin{equation*}
  \beta^{\underline{n}_{\gamma}^{\mathbf{d}}}\alpha^{\underline{n}_{\gamma}^{\mathbf{d}}}\gamma \beta^{\underline{n}_{\gamma}^{\mathbf{r}}}\alpha^{\underline{n}_{\gamma}^{\mathbf{r}}}=\gamma.
\end{equation*}
This implies that
\begin{equation*}
  \beta^{\underline{n}_{\gamma}^{\mathbf{d}}}\varepsilon_0\alpha^{\underline{n}_{\gamma}^{\mathbf{r}}}= \beta^{\underline{n}_{\gamma}^{\mathbf{d}}}\alpha^{\underline{n}_{\gamma}^{\mathbf{d}}}\gamma \beta^{\underline{n}_{\gamma}^{\mathbf{r}}}\alpha^{\underline{n}_{\gamma}^{\mathbf{r}}}=\gamma,
\end{equation*}
and hence the statement of our lemma holds.
\end{proof}

\begin{lemma}\label{lemma-3.13}
Let $A$ be a generating set of the monoid $\mathbf{I}\mathbb{N}_{\infty}$ and $A_\mathscr{C}^\circ$ be a minimal finite subset of $A$ such that $\mathscr{C}_{\mathbb{N}}$ is a subsemigroup of $\langle A\rangle$. Then for any integer $k\geqslant 2$ and any representation $\varepsilon^{[k]}=\gamma_1\cdots\gamma_s$, $\gamma_1,\ldots,\gamma_s\in A$, there exist finitely many $\gamma_1^*,\ldots,\gamma_s^*\in A\cup \mathscr{C}_{\mathbb{N}}$ such that
\begin{equation}\label{eq3-1}
\varepsilon^{[k]}=\gamma_1^*\cdots\gamma_s^* \quad \hbox{and either} \quad \gamma_j^*=\gamma_j\in A\setminus \mathbf{I}\mathbb{N}_{\infty}^{\boldsymbol{g}[k-1]} \; \hbox{or} \; \gamma_j^*\in \mathscr{C}_{\mathbb{N}}, \quad \hbox{for} \quad j=1,\ldots,s.
\end{equation}
 Moreover, if $\gamma_j^*\in \mathscr{C}_{\mathbb{N}}$,  then there exist $\delta_{j,1},\ldots,\delta_{j,p_j}\in A_\mathscr{C}^\circ$ such that $\gamma_j^*=\delta_{j,1}\cdots\delta_{j,p_j}$ for some positive integer $p_j$.
\end{lemma}

\begin{proof}
Fix any integer $k\geqslant 2$ and suppose that $\varepsilon^{[k]}=\gamma_1\cdots\gamma_s$ for some $\gamma_1,\ldots,\gamma_s\in A$.

The definitions of the idempotent $\varepsilon^{[k]}$ and composition of partial maps (see \cite[Section~1.1]{Lawson-1998}) imply that either $\operatorname{dom}\gamma_1=\mathbb{N}$ or $\operatorname{dom}\gamma_1=\operatorname{dom}\varepsilon^{[k]}$, because the set $\mathbb{N}\setminus\operatorname{dom}\varepsilon^{[k]}$ is a singleton. If $\operatorname{dom}\gamma_1=\mathbb{N}$, then by Lemma~1 of \cite{Gutik-Savchuk-2018}, $\gamma_1$ is the partial shift of integers, and hence $\gamma_1\in \mathscr{C}_{\mathbb{N}}$. If $\operatorname{dom}\gamma_1=\operatorname{dom}\varepsilon^{[k]}$, then similar arguments imply that $\gamma_1$ is the partial shift of the set $\mathbb{N}\setminus\{k\}$. In both cases we put $\gamma_1^*=\gamma_1$.

Next we consider the element $\gamma_2$. The definition of the monoid $\mathbf{I}\mathbb{N}_{\infty}$ and Lemma~1 of \cite{Gutik-Savchuk-2018} imply that $(\operatorname{dom}\varepsilon^{[k]})\gamma_1\subseteq \operatorname{dom}\gamma_2$.

Suppose that $n_{\gamma_2}^{\mathbf{d}}-\underline{n}_{\gamma_2}^{\mathbf{d}}\geqslant k$. Then one of the following cases holds:
\begin{equation*}
  n_{\gamma_2}^{\mathbf{d}}=(k+1)\gamma_1 \qquad \hbox{or} \qquad n_{\gamma_2}^{\mathbf{d}}\leqslant(1)\gamma_1.
\end{equation*}
In the first case we have that $\left\{(i)\gamma_1\colon i=1,\ldots,k-1\right\}\subseteq \operatorname{dom}\gamma_2$ and hence we put $\gamma_2^*=\gamma_2$. In the second case by Lemma~1 of \cite{Gutik-Savchuk-2018}, $\gamma_2$ is the partial shift of integers, and we put $\gamma_2^*=\beta^{(1)\gamma_1}\alpha^{(1)\gamma_1\gamma_2}$. It it obvious that in both cases we have that $\gamma_1\gamma_2=\gamma_1^*\gamma_2^*$.

Suppose that $n_{\gamma_2}^{\mathbf{d}}-\underline{n}_{\gamma_2}^{\mathbf{d}}< k$. Then the equality $\varepsilon^{[k]}=\gamma_1\gamma_2\cdots\gamma_s$ implies that $n_{\gamma_2}^{\mathbf{d}}\leqslant(1)\gamma_1$, and hence the above presented arguments imply that $\gamma_1\gamma_2=\gamma_1^*\gamma_2^*$, where $\gamma_2^*=\beta^{(1)\gamma_1}\alpha^{(1)\gamma_1\gamma_2}$.

Using induction up to $s$ in the similar way we obtain the requested representation of the idempotent $\varepsilon^{[k]}=\gamma_1^*\cdots\gamma_s^*$ in form \eqref{eq3-1}. Also, since $k\notin\operatorname{dom}\varepsilon^{[k]}$, there exists a smallest positive integer $j\leqslant s$ such that $(1)\gamma_1\cdots\gamma_{j-1}\notin \operatorname{dom}\gamma_j$. This completes the first statement of the lemma. The second statement is obvious and  follows from Lemma~\ref{lemma-3.11}.
\end{proof}

\begin{theorem}\label{theorem-3.14}
Let $A$ be any infinite subset of $\mathbf{I}\mathbb{N}_{\infty}$ generating the monoid $\mathbf{I}\mathbb{N}_{\infty}$. Then there exists no a minimal subset $B\subseteq A$  generating $\mathbf{I}\mathbb{N}_{\infty}$.
\end{theorem}

\begin{proof}
By Lemma~\ref{lemma-3.11} there exists a minimal finite subset $A_\mathscr{C}^\circ$ of $A$ such that $\mathscr{C}_{\mathbb{N}}\subseteq \left\langle A_\mathscr{C}^\circ\right\rangle$. Put $j_1=2$. Since $\varepsilon^{[j_1]}=\gamma_1\cdots\gamma_{s_1}$ for some $\gamma_1,\ldots,\gamma_{s_1}\in A$, there exists the smallest positive integer $k_1$ such that $n_{\gamma_i}^{\mathbf{d}}-\underline{n}_{\gamma_i}^{\mathbf{d}}\leqslant k_1$ for any $i=1,\ldots,s_1$ and $n_{\gamma}^{\mathbf{d}}-\underline{n}_{\gamma}^{\mathbf{d}}\leqslant k_1$ for any $\gamma\in A_\mathscr{C}^\circ$. By Lemma~\ref{lemma-3.12}, $\mathbf{I}\mathbb{N}_{\infty}^{\boldsymbol{g}[k_1]}$ is a subsemigroup of $\left\langle A_\mathscr{C}^\circ\cup\left\{\gamma_1,\ldots,\gamma_{s_1}\right\}\right\rangle$.

Put $j_2=k_1+1$. Then by Lemmas~\ref{lemma-3.4},~\ref{lemma-3.5}, \ref{lemma-3.6} we have that
\begin{equation*}
\varepsilon^{[j_2]}\notin \left\langle A_\mathscr{C}^\circ\cup\left\{\gamma_1,\ldots,\gamma_{s_1}\right\}\right\rangle.
\end{equation*}

Suppose that $\varepsilon^{[j_2]}=\gamma_{s_1+1}\cdots \gamma_{s_2}$ for some $\gamma_{s_1+1},\ldots, \gamma_{s_2}\in A$, where $s_1+1\leqslant s_2$. By Lemma~\ref{lemma-3.13} there exist finitely many $\gamma_{s_1+1}^*,\ldots, \gamma_{s_2}^*\in A\cup \mathscr{C}_{\mathbb{N}}$ such that
\begin{equation*}
\varepsilon^{[k]}=\gamma_{s_1+1}^*\cdots\gamma_{s_2}^* \; \hbox{and  either} \; \gamma_j^*=\gamma_j\in A\setminus \mathbf{I}\mathbb{N}_{\infty}^{\textbf{\emph{g}}[k-1]} \; \hbox{or} \; \gamma_j^*\in \mathscr{C}_{\mathbb{N}}, \; \hbox{for} \; j=s_1+1,\ldots,s_2.
\end{equation*}
The second statement of Lemma~\ref{lemma-3.13} and  Lemma~\ref{lemma-3.12} imply that
\begin{equation*}
\mathbf{I}\mathbb{N}_{\infty}^{\boldsymbol{g}[j_2]}\subseteq \left\langle A_\mathscr{C}^\circ\cup\left\{\gamma_{s_1+1},\ldots,\gamma_{s_2}\right\}\right\rangle\subseteq \left\langle A\setminus\mathbf{I}\mathbb{N}_{\infty}^{\boldsymbol{g}[j_2]} \cup A_\mathscr{C}^\circ \right\rangle.
\end{equation*}

Next, if we repeat the above presented construction infinitely many times, then we obtain an increasing sequence of positive integers $\left\{j_p\right\}_{p\in\mathbb{N}}$  such that
\begin{equation*}
  \mathbf{I}\mathbb{N}_{\infty}^{\boldsymbol{g}[j_p]}\subseteq \left\langle A\setminus\mathbf{I}\mathbb{N}_{\infty}^{\boldsymbol{g}[j_p]} \cup A_\mathscr{C}^\circ \right\rangle \qquad \hbox{for any } \quad j_p.
\end{equation*}
Since \begin{equation*}
\mathbf{I}\mathbb{N}_{\infty}^{\boldsymbol{g}[0]}= \mathbf{I}\mathbb{N}_{\infty}^{\boldsymbol{g}[1]}\subsetneqq \mathbf{I}\mathbb{N}_{\infty}^{\boldsymbol{g}[2]}\subsetneqq \mathbf{I}\mathbb{N}_{\infty}^{\boldsymbol{g}[1]}\subsetneqq \cdots \subsetneqq \mathbf{I}\mathbb{N}_{\infty}^{\boldsymbol{g}[k]}\subsetneqq \cdots \subset \mathbf{I}\mathbb{N}_{\infty} \end{equation*}
and $\mathbf{I}\mathbb{N}_{\infty}=\displaystyle\bigcup_{i=1}^\infty\mathbf{I}\mathbb{N}_{\infty}^{\boldsymbol{g}[i]}$, Lemma~\ref{lemma-3.12} implies that the set $A$ does not contain  a minimal subset $B\subseteq A$ which generates the monoid $\mathbf{I}\mathbb{N}_{\infty}$.
\end{proof}

Theorem~\ref{theorem-3.14} implies the following corollary.

\begin{corollary}\label{corollary-3.15}
The monoid $\mathbf{I}\mathbb{N}_{\infty}$ does not contains a minimal generating set.
\end{corollary}

\medskip
\section*{\textbf{Acknowledgements}}
The author acknowledges Taras Banakh, Alex Ravsky and the referee for useful important comments and suggestions.

\end{document}